\documentclass[12pt]{amsart}

\usepackage{amssymb,latexsym}
\usepackage{enumerate}
\usepackage[hyperfootnotes=false]{hyperref}
\usepackage{fullpage}
\usepackage{soul}

\usepackage{fixme}
\fxsetup{
    status=draft,
    author=,
    layout=margin,
    theme=color
}

\makeatletter \@namedef{subjclassname@2010}{%
  \textup{2010} Mathematics Subject Classification}
\makeatother

\newcounter{thm} \numberwithin{thm}{section}
\newtheorem{Theorem}[thm]{Theorem}

\newtheorem{Lemma}[thm]{Lemma}

\newtheorem{Claim}[thm]{Claim}

\renewcommand{\emptyset}{\varnothing}

\newcommand{\thistheoremname}{}
\newtheorem*{genericthm*}{\thistheoremname}
\newenvironment{namedthm*}[1]
  {\renewcommand{\thistheoremname}{#1}%
   \begin{genericthm*}}
  {\end{genericthm*}}
\usepackage{tikz}
\usetikzlibrary{decorations.pathreplacing}
\tikzset{mybrace/.style={decoration={brace,raise=1.8mm},decorate}}
\tikzset{mybracedown/.style={decoration={brace,mirror,raise=1.8mm},decorate}}
\usetikzlibrary{math}
\usetikzlibrary{patterns}
\usetikzlibrary{matrix,backgrounds}

\date{}

\author[K. Bhowmick]{Krishnendu Bhowmick} \address{Johann Radon Institute for Computational and Applied Mathematics\\
Linz, Austria}
\email{Krishnendu.Bhowmick@oeaw.ac.at}

\author[B. Lund]{Ben Lund} 
\address{Institute for Basic Science (IBS)\\
Daejeon, South Korea}
\email{benlund@ibs.re.kr}

\author[O. Roche-Newton]{Oliver Roche-Newton} \address{Institute for Algebra, Johannes Kepler Universit\"{a}t\\
Linz, Austria}
\email{o.rochenewton@gmail.com}

\begin{document}

\baselineskip=17pt

\title{Large convex sets in difference sets}

\begin{abstract} 

We give a construction of a convex set $A \subset \mathbb R$ with cardinality $n$ such that $A-A$ contains a convex subset with cardinality $\Omega (n^2)$. We also consider the following variant of this problem: given a convex set $A$, what is the size of the largest matching $M \subset A \times A$ such that the set
\[
\{ a-b : (a,b) \in M \}
\]
is convex? We prove that there always exists such an $M$ with $|M| \geq \sqrt n$, and that this lower bound is best possible, up a multiplicative constant.
\end{abstract}

\date{}
\maketitle

\section{Introduction}

This paper is concerned with the existence or non-existence of convex sets in difference sets. A set $A \subset \mathbb R$ is said to be convex if its consecutive differences are strictly increasing. That is, writing $A= \{a_1 < a_2 < \dots < a_n \}$, $A$ is convex if
\[
a_{i+1} - a_i > a_i - a_{i-1}
\]
holds for all $i=2,\dots, n-1$. Most research on convex sets comes in the context of sum-product theory, and one may think of the notion of a convex set as a generalisation of a set with multiplicative structure. For instance, it is known than convex sets determine many distinct sums and differences. In particular, it was proven in \cite{SS} that the bound\footnote{Here and throughout this paper, the notation $\gg$ is used to absorb a multiplicative constant. That is $X \gg Y$ denotes that there exists an absolute constant $C>0$ such that $X \geq CY$. The notation $Y \ll X$ has the same meaning.}
\[
|A-A| \gg |A|^{8/5-o(1)}
\]
holds for any convex set $A$. Here, $A-A:=\{a-b : a,b \in A \}$ denotes the difference set determined by $A$. This results captures the vague notion that convex sets cannot be additively structured, and there has been considerable effort expended to quantify and apply this idea in various way, see for instance \cite{ENR}, \cite{HRNR} and \cite{SW}.

Given a finite set $B \subset \mathbb R$, define
\[
\mathcal C(B):= \max_{A \subset B : A \text{ is convex}} |A|. 
\]
That is, $\mathcal C(B)$ denotes the size of the largest convex subset of $B$. The first question that we consider is the following: given a convex set $A \subset \mathbb R$, what can we say about the possible value of $\mathcal C(A-A)$? A first observation is that
\begin{equation} \label{easylower}
|\mathcal C(A-A)| \geq |A|,
\end{equation}
as can be seen by considering the convex set $A-a \subset A-A$, where $a$ is an arbitrary element of $A$. There are some simple constructions showing that the lower bound \eqref{easylower} is optimal up to a multiplicative constant; for instance, we can take $A= \{ i^2 : 1 \leq i \leq n \}$. 

In this paper, we give a construction of a convex set $A$ whose difference set contains a very large convex set.

\begin{Theorem} \label{thm:main}
For all $n \in \mathbb N$, there exists a convex set $A \subset \mathbb R$ with $|A|=n$ such that
$A-A$ contains a convex subset $S$ with cardinality $|S| \gg n^2$.
\end{Theorem}

Using the notation introduced earlier, Theorem \ref{thm:main} states that there exists a convex set $A$ such that $\mathcal C(A-A) \gg |A|^2$. This result shares some similarities with main result of \cite{RZ}, where it was established that there exists a set $A \subset \mathbb R$ such that $A+A$ contains a convex subset with cardinality $\Omega(|A|^2)$. The main qualitative difference is that we have the additional restriction that the set $A$ is also assumed to be convex. Also, Theorem \ref{thm:main} provides a convex subset of the difference set, rather than the sum set.

 The simple construction giving rise to the lower bound \eqref{easylower} feels like something of a cheat, and so we consider a variant of this problem where we make a further restriction concerning the origin of the convex subset of a difference set. A set $M \subset A \times A$ is a \textit{matching} if the elements of $M$ are pairwise disjoint. Given a matching $M \subset A \times A$, define the restricted difference set
 \[
 A-_M A:= \{ a-b : (a,b) \in M \}.
 \]

 Define
\[
\mathcal {CM}(A):= \max_{M \subset A \times A : M \text{ is a matching and } A-_M A \text{ is convex.}} |M|. 
\]
That is, $\mathcal {CM}(A)$ denotes the size of the largest matching on $A$ which gives rise to a convex subset of $A-A$.

Now, we ask a similar question for this quantity: given a convex set $A \subset \mathbb R$, what can we say about the size of $\mathcal {CM}(A)$? In particular, how small can this quantity be? Should we expect an analogue of the bound \eqref{easylower} if we rule out this simple construction? In this paper we answer this question by giving the following two complimentary results, showing that $\mathcal {CM}(A) \geq \sqrt { |A|}$ and that this bound is optimal up to a multiplicative constant.

\begin{Theorem} \label{thm:main2}
Let $n \in \mathbb N$ be sufficiently large and suppose that $A \subset \mathbb R$ is a convex set with cardinality $n$. Then there exists a matching $M \subset A \times A$ such that $|M| \geq \sqrt n$ and $A-_M A$ is convex.
\end{Theorem}

\begin{Theorem} \label{thm:main3}
    For all $n \in \mathbb N$, there exists a convex set $A \subset \mathbb R$ with cardinality $n$ such that, if $M \subset A \times A$ is a matching and $A-_M A$ is convex then $|M| \ll \sqrt n$.
\end{Theorem}

\section{Proof of Theorem \ref{thm:main}}

Define
\[
a_i=i + c_1i^2 + c_2i^3,
\]
where
\[
c_1= \frac{75}{n^2} \,\,\,\, \text{  and } \,\,\,\, c_2 = \frac{1}{n^5}.
\]
Assume that $n$ is a sufficiently large multiple of $100$. This assumption is made only to simplify the notation slightly, and can be easily removed at the price of introducing some floor and ceiling functions to the calculations.
Define $A= \{a_1< \dots <a_n \}$. Observe that $A$ is convex. Indeed, the sequence $a_{i+1}-a_i$ is increasing, as can be seen by calculating that
\[
a_{i+1}-a_i= 1 +c_1(2i+1) +c_2(3i^2+3i+1).
\]

For each integer $k \in [0.009 n , 0.01 n]$ define $D_k$ to be the set of $k$th differences
\[
D_k:= \{ a_{i+k}-a_i : 1 \leq i \leq  0.99 n \} .
\]
The set $D_k$ is also convex. Indeed, let $d_i^{(k)}:= a_{i+k}-a_i$ denote the $i$th element of $D_k$. Then the sequence $d_{i+1}^{(k)}- d_{i}^{(k)} $ increases with $i$. This can be seen by observing that
\begin{equation} \label{bidef}
d_i^{(k)}= a_{i+k}-a_i = k + c_1(2ki+k^2) +c_2(3i^2k+3ik^2+k^3),
\end{equation}
and hence
\[
d_{i+1}^{(k)}- d_i^{(k)}= 2c_1k+3c_2k^2+3c_2k+6c_2ki
\]
increases with $i$.

We will find a large convex subset of $A-A$ by efficiently gluing together consecutive convex sets $D_k$. We will make use of the following observation from \cite{RZ}.

\begin{Lemma} \label{lem:RZ}
Suppose that $A= \{a_1< \dots <a_n \}$ and $B = \{b_1< \dots <b_m \}$ are convex sets. Suppose that there exist $1 \leq i \leq m$ and $1 \leq  j \leq n$ such that
\[
b_i \leq a_j < a_{j+1} \leq b_{i+1}.
\]
Then the set
\[
\{ a_1 < \dots < a_j < b_{i+1} < b_{i+2} < \dots < b_m \}
\]
is convex.
\end{Lemma}

Before we get into the details of the proof of Theorem \ref{thm:main}, which involves some rather tedious calculations, let us take a moment to try and explain the idea behind it, with the help of some pictures.

Firstly, we note that, although the sets $D_k$ are convex, they are only \textit{slightly convex}, in the sense that, if we zoom out and take a look at $D_k$, it appears to resemble an arithmetic progression with common difference $2c_1k$. Note also that this common difference increases slightly as $k$ increases.

 \begin{figure}[ht]

\centering

\begin{tikzpicture} [scale=3]
  % Draw the number line
  \draw (11.9,0) -- (16.1,0);
  % Add labels for the values
  \node[below] at (12,0) {$d_1^{(10)}$};
  \node[below] at (14.00008,0) {$d_2^{(10)}$};
  \node[below] at (16.00050,0) {$d_3^{(10)}$};
  %\node[above] at (11.1,0) {$d_4^{(1)}$};
 % \node[below] at (1.1111,0) {11111};
 % \node[below] at (12,0) {100000}; 
  % Mark the values on the number line
  \draw[fill] (12,0) circle (0.4pt);
  \draw[fill] (14.00008,0) circle (0.4pt);
 \draw[fill] (16.00050,0) circle (0.4pt);
 %\draw[fill] (11.1,0) circle (1pt);
\end{tikzpicture}

\caption{This picture shows the first three elements of $D_{10}$ after setting $n=10000$. The three elements form a convex set, but to the naked eye they appear to be arranged in an arithmetic progression.}

\end{figure}
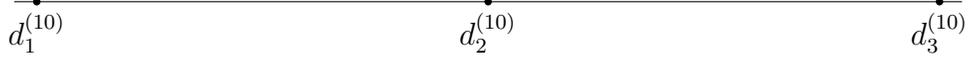

The other important feature of this construction is that we have chosen the parameters in such a way that the $D_k$ have convenient overlapping properties. In particular, each $D_k$ has diameter approximately $\frac{3}{2}$, and starts at $k$. In particular, this means that neighbouring $D_k$ have a significant overlap, but also that each $D_k$ takes sole ownership of a section of the real line.

 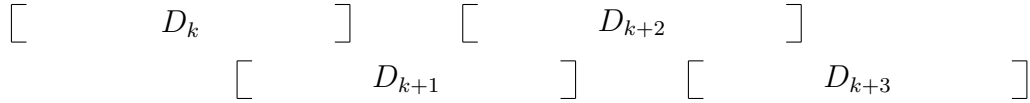
\begin{figure}[ht]

\centering

\begin{tikzpicture} [scale=1]
  % Draw the number line
 % \draw (6,0) -- (21,0);
  % Add labels for the values
  
 % \node[below] at (6.8,-0.2) {$D_1$};
%   \node[below] at (13.4,-0.2) {$D_2$};
  %  \node[below] at (20,-0.2) {$D_3$};
  
  \node at (2.25,1) {$D_k$};
\node at (5.25,0.25) {$D_{k+1}$};
  \node at (8.25,1) {$D_{k+2}$};
  \node at (11.25,0.25) {$D_{k+3}$};
 % \node[below] at (1.11,0) {11100};
 % \node[below] at (1.111,0) {11110};
 % \node[below] at (1.1111,0) {11111};
 % \node[below] at (12,0) {100000}; 
  % Mark the values on the number line

\draw (0,1.25) -- (0,0.75);
\draw (4.5,1.25) -- (4.5,0.75);
\draw (0,1.25) -- (0.2,1.25);
\draw (0,0.75) -- (0.2,0.75);
\draw (4.5,1.25) -- (4.3,1.25);
\draw (4.5,0.75) -- (4.3,0.75);

\draw (3,0.5) -- (3,0);
\draw (7.5,0.5) -- (7.5,0);
\draw (3,0.5) -- (3.2,0.5);
\draw (3,0) -- (3.2,0);
\draw (7.5,0.5) -- (7.3,0.5);
\draw (7.5,0) -- (7.3,0);

\draw (6,1.25) -- (6,0.75);
\draw (10.5,1.25) -- (10.5,0.75);
\draw (6,1.25) -- (6.2,1.25);
\draw (6,0.75) -- (6.2,0.75);
\draw (10.5,1.25) -- (10.3,1.25);
\draw (10.5,0.75) -- (10.3,0.75);

\draw (9,0.5) -- (9,0);
\draw (13.5,0.5) -- (13.5,0);
\draw (9,0.5) -- (9.2,0.5);
\draw (9,0) -- (9.2,0);
\draw (13.5,0.5) -- (13.3,0.5);
\draw (13.5,0) -- (13.3,0);

\end{tikzpicture}

\caption{This diagram illustrates the intersection pattern of the sets $D_k$.}

\end{figure}

We can use this setup to form a convex set by gluing together consecutive $D_k$. In the region where $D_k$ and $D_{k+1}$ overlap, the elements of $D_k$ are slightly more dense (because the common difference of the approximate arithmetic progression is smaller). This ensures that there exist two consecutive elements of $D_k$ in this region, which allows for an application of Lemma \ref{lem:RZ}. Meanwhile, the existence of the non-overlapping region ensures that this glued set contains many elements of $D_k$ for each $k$.

Now we come to the formal details. We use the notation $d_{min}^{(k)}$ for the smallest element of $D_k$ and $d_{max}^{(k)}$ for the largest element of $D_k$. Note that
\begin{align}
    d_{min}^{(k)} & = d_1^{(k)} =  k + c_1(2k+k^2) +c_2(3k+3k^2+k^3), \text{ and} \label{bmin}
    \\  d_{max}^{(k)} & = d_{0.99n}^{(k+1)} = k + c_1\left(2k\frac{99n}{100}+k^2\right) +c_2\left(3\left(\frac{99n}{100} \right)^2k+3\frac{99n}{100}k^2+k^3\right).
    \label{bmax}
\end{align}

Observe that each $D_k$ is contained in the closed interval $[b_{min}^{(k)}, b_{max}^{(k)}]$. We will prove the following two facts about the intersection properties of these intervals.

\begin{Claim} \label{claim1}
    For each $k \in [0.009 n , 0.01 n] $, there exist $1 \leq i,j \leq 0.99n$ such that
    \[
    d_i^{(k+1)} \leq d_j^{(k)} < d_{j+1}^{(k)} \leq d_{i+1}^{(k+1)} .
    \]
\end{Claim}

\begin{Claim} \label{claim2}
    For each $k \in  [0.009 n , 0.01 n]$, there are at least $Cn$ elements of $D_k$ in the interval $(d_{max}^{(k-1)}, d_{min}^{(k+1)})$, where $C>0$ is an absolute constant.
\end{Claim}

Once we have proved these two claims, the proof will be finished. Indeed, we can use Claim \ref{claim1} together with Lemma \ref{lem:RZ} to glue together consecutive convex sets $D_k$ to form a set 
\[
S \subset\bigcup_{k=0.009 n}^{0.01n} D_k \subset A-A.
\]
Claim \ref{claim2} guarantees that, for each $k \in  [0.009 n , 0.01 n]$, there at least $Cn$ elements in $D_k \cap S$ that do not appear in $D_j \cap S$ for any $j \neq k$. This implies that
\[
|S| \geq (Cn) \cdot (0.001 n ) \gg n^2.
\]

It remains to prove the two claims.

\begin{proof}[Proof of Claim \ref{claim1}]
We will show that the interval $I=[d_{min}^{(k+1)},d_{max}^{(k)}]$ contains at least two more elements of $D_k$ than it does of $D_{k+1}$. It then follows that there must exist two consecutive elements of $D_k$ in this interval, which then implies the existence of the claimed configuration
\[
    d_i^{(k+1)} \leq d_j^{(k)} < d_{j+1}^{(k)} \leq d_{i+1}^{(k+1)} .
 \]

We begin by establishing a lower bound for $|D_k \cap I|$, namely
\begin{align*}
|D_k \cap I| &= | \{ i \in [  0.99n  ] : d_i^{(k)} \geq d_{min}^{(k+1)} \}|
\\& \geq 0.99n - | \{ i \in \mathbb N : d_i^{(k)} <  d_{min}^{(k+1)} \}|.
\end{align*}
We need an upper bound for $| \{ i \in \mathbb N : d_i^{(k)} <  d_{min}^{(k+1)} \}|$. To this end,
\begin{align} \label{Bkeq}
&| \{ i \in \mathbb N : d_i^{(k)} <  d_{min}^{(k+1)} \}|  \nonumber
\\ & \leq 2 + |\{i > 2: d_i^{(k)} < d_{min}^{(k+1)}\}| \nonumber 
\\ & = 2 + | \{ i > 2 : 2c_1ki +c_2(3i^2k+3ik^2) <  1 + c_1(4k+3) +c_2(6k^2+12k+7) \}| 
\\& \leq 2 + | \{ i >2 : 2c_1ki  <  1 + c_1(4k+3)  \}| \nonumber
\\& \leq 2 + \frac{1+c_1(4k+3)}{2c_1k} = \frac{1+c_1(8k+3)}{2c_1k}. \nonumber
\end{align}
%\ben{There was a slight issue here in that the inequality on the $c_2 \ldots$ part only holds if $i > 2$, but then we win big in the next line. I'm not sure what the best way to handle it is, I don't really like what I did above. Maybe the best way is to go back to what you had before, and add a note before the sequence of inequalities about what we plan to do, so that someone who checks line-by-line doesn't get confused. Maybe we can use a final upper bound like $0.75n$?}
Therefore,
\begin{equation} \label{BkIlower}
    |D_k \cap I| \geq  \frac{1.98 n \cdot c_1k - 8c_1k - 1 - 3c_1 }{2c_1k} .
\end{equation}
We will compare the bound in \eqref{BkIlower} with an upper bound for $|D_{k+1} \cap I|$, which we deduce now. Observe that
\begin{align*}
|D_{k+1} \cap I| &= | \{ i \in [  0.99n  ] : d_i^{(k+1)} \leq d_{max}^{(k)} \}|
\\& =  | \{ i \in\mathbb N: 1 + c_1(2(k+1)i+2k+1) +c_2(3i^2(k+1)+3i(k+1)^2+3k^2+3k+1 )
\\ \,\,\, & \leq  c_12k\frac{99n}{100} +c_2\left(3\left(\frac{99n}{100} \right)^2k+3\frac{99n}{100}k^2\right) \} |
\\& \leq \left |  \left \{ i \in \mathbb N : 1 + c_1(2(k+1)i+2k+1) 
\leq  2c_1k\frac{99n}{100} +c_2\left(3\left(\frac{99n}{100} \right)^2k+3\frac{99n}{100}k^2\right) \right \} \right  |.
\end{align*}
Note that the term which involves the multiple of $c_2$ in the previous step is at most $1/n^2$. Therefore, this term is less than $c_1$, which allows us to write
\begin{align}
|D_{k+1} \cap I| &\leq   \left |  \left \{ i \in \mathbb N : 1 + 2c_1(k+1)i+2c_1k
\leq 2 c_1k\frac{99n}{100} \right  \}  \right | \nonumber
\\& =  \left |  \left \{ i \in \mathbb N : i
\leq  \frac{2c_1k\left (\frac{99n}{100}-1 \right ) -1}{2c_1(k+1)} \right  \}  \right | \nonumber
\\ & \leq \frac{2c_1k\left (\frac{99n}{100}-1 \right ) -1}{2c_1(k+1)}. \label{Bk+1Iupper}
\end{align}
It remains to show that the lower bound given in \eqref{BkIlower} is at least as big as the upper bound given in \eqref{Bk+1Iupper}, plus two. That is, we need to show that
\begin{align*} 
    \frac{1.98 n \cdot c_1k - 8c_1k - 1 - 3c_1 }{2c_1k} & \geq 2 + \frac{2c_1k\left (\frac{99n}{100}-1 \right ) -1}{2c_1(k+1)} 
    \\ & =  \frac{4c_1(k+1) + 2c_1k\left (\frac{99n}{100}-1 \right ) -1}{2c_1(k+1)},
\end{align*}
which simplifies to give
\[
 \frac{1.98 n \cdot c_1k - 8c_1k - 1 - 3c_1 }{k} \geq \frac{4c_1(k+1) + 2c_1k\left (\frac{99n}{100}-1 \right ) -1}{k+1},
\]
and eventually
\[
1.98 nc_1k \geq 10c_1k^2+15c_1k+3c_1+1.
\]
Since we have $0.009n \leq k \leq 0.01n$, it would be sufficient to prove that
\[
1.98 n \cdot c_1 \cdot 0.009 n \geq 10c_1 \cdot (0.01n)^2 +15c_1 \cdot (0.01 n) +3c_1+1. 
\]
Substituting in the definition of $c_1$, the previous inequality becomes
\[
\frac{198 \cdot 75 \cdot 9}{100 \cdot 1000} \geq \frac{75}{1000} + \frac{15 \cdot 75}{100 n} + \frac{3 \cdot 75}{n^2} + 1.
\]
This is equivalent to the inequality
\[
0.2615 \geq  \frac{15 \cdot 75}{100 n} + \frac{3 \cdot 75}{n^2},
\]
which is valid for $n$ sufficiently large.

\end{proof}

\begin{proof}[Proof of Claim \ref{claim2}]
We will give an upper bound for $|D_k \cap(-\infty, d_{max}^{(k-1)}]|$ and a lower bound for $|D_k \cap( - \infty, d_{min}^{(k+1)})|$, and combine these bounds to deduce the claimed lower bound for $|D_k \cap ( d_{max}^{(k-1)}, d_{min}^{(k+1)})|$. To upper bound $|D_k \cap(-\infty, d_{max}^{(k-1)}]|$, we make use of \eqref{Bk+1Iupper} and deduce that
\begin{align*}
    |D_k \cap(-\infty, d_{max}^{(k-1)}]| & = |D_k \cap [d_{min}^{(k)},d_{max}^{(k-1)}]|
    \\& \leq \frac{2c_1(k-1)\left (\frac{99n}{100}-1 \right ) -1}{2c_1k}
    \\& \leq \frac{1}{2} \left ( \frac{150}{n^2} \cdot (0.01n) \cdot n -1 \right ) \cdot  \left( \frac{1000}{9n}\cdot  \frac{n^2}{75} \right ) = \frac{10n}{27}.
\end{align*}

Next, we present a lower bound for $|D_k \cap( - \infty, d_{min}^{(k+1)})|= |D_k \cap [d_{min}^{(k)},d_{min}^{(k+1)})|$. Indeed,
\begin{align} \label{nearly}
    |D_k \cap [d_{min}^{(k)},d_{min}^{(k+1)})| & =| \{ i \in [ 0.99n ] : d_i^{(k)} < d_{min}^{(k+1)})\}| \nonumber
    \\& = | \{ i \in  [ 0.99n ]: 2c_1ki +c_2(3i^2k+3ik^2) <  1 + c_1(4k+3) +c_2(6k^2+12k+7) \}| \nonumber
    \\& \geq | \{ i  \in [ 0.99n ] : 2c_1ki +c_2(3i^2k+3ik^2) <  1 + c_1(4k+3) \}| \nonumber
    \\& \geq | \{ i  \in [ 0.99n ] : 2c_1ki  <  0.99 + c_1(4k+3) \}|. 
\end{align}
The last inequality above uses the fact that the term $c_2(3i^2k+3ik^2)$ is at most $0.01$, provided that $n$ is sufficiently large. Therefore,
\begin{align*}
  |D_k \cap [d_{min}^{(k)},d_{min}^{(k+1)})|  &\geq | \{ i  \in \mathbb N: 2c_1ki  <  0.99 + c_1(4k+3) \}|
  \\ & \geq \frac{0.99 + c_1(4k+3)}{2c_1k} -1
    \\ & \geq \frac{0.99}{2c_1k} -1
  \\& \geq \frac{99}{100} \cdot \frac{2n}{3} -1 =\frac{66}{100}n-1 \geq \frac{65}{100}n=\frac{13}{20}n.
\end{align*}
By combining this inequality with \eqref{nearly}, it follows that
\[
|D_k \cap (d_{max}^{(k-1)}, d_{min}^{(k+1)})| \geq \frac{13}{20}n - \frac{10}{27}n = \frac{151}{540}n,
\]
as required.
\end{proof}

\section{Matchings}

%\begin{Theorem} \label{thm:match1}
%Let $A \subset \mathbb R$ be a convex set with cardinality $n$. Then there exists a matching $M \subset A \times A$ such that $|M| \geq \sqrt n$ and $A-_M A$ is convex.
%\end{Theorem}

\begin{proof}[Proof of Theorem \ref{thm:main2}]
    Again, write $A= \{ a_1<a_2< \dots < a_n \}$ and $d_i=a_{i+1}-a_i$. Since $A$ is convex, we have
    \[ 
    d_1 < d_2 < \dots <d_{n-1}.
    \]
    For convenience, we the shorthand use $k:= \lceil \sqrt n \rceil$. The matching $M$ is given by
    \begin{align*}
    M&= \left \{ (a_{k+1},a_{k+2}), (a_k, a_{k+4}), (a_{k-1}, a_{k+7}), (a_{k-2}, a_{k+11}), \dots ,\left(a_1,a_{k +1+\frac{k(k+1)}{2}}\right ) \right \} 
    \\& = \left \{ \left (a_{k+1-i}, a_{k+1+\frac{i(i+1)}{2}} \right ) : 1 \leq i \leq k \right \}.
    \end{align*}
    The matching $M$ has cardinality $k \geq \sqrt n$, and the choice of parameters ensures that it is indeed a well-defined subset of $A \times A$, provided that $n$ is sufficiently large. To verify this, we just need to check that $k +1+\frac{k(k+1)}{2} \leq n $, which holds comfortably for $n \geq 36$.

    It remains to check that $A-_M A$ is convex. Let 
    \[
    e_i:=a_{k +1+\frac{i(i+1)}{2}} - a_{k+1-i}
    \]
    denote the $i$th element of $A-_M A$.  We need to check that $e_{i+1}-e_{i} > e_{i} - e_{i-1}$ holds for all $2 \leq i \leq k-1$. A telescoping argument gives
    \begin{align} \label{eqed}
    e_i&= (a_{k+2-i} - a_{k+1-i}) + (a_{k+3-i} - a_{k+2-i}) + \dots + \left (a_{k +1+\frac{i(i+1)}{2}} - a_{k +\frac{i(i+1)}{2}} \right ) \nonumber
    \\& = d_{k+1-i}+d_{k+2-i} + \dots + d_{k+\frac{i(i+1)}{2}},
    \end{align}
   and therefore
   \begin{align*}
   e_{i+1} - e_i & = d_{k+1-(i+1)}+d_{k+2-(i+1)} + \dots + d_{k+\frac{(i+1)(i+2)}{2}}
   \\& - \left (d_{k+1-i}+d_{k+2-i} + \dots + d_{k+\frac{i(i+1)}{2}} \right )
   \\& =  d_{k-i}+ d_{k+\frac{i(i+1)}{2}+1}+ d_{k+\frac{i(i+1)}{2}+2} + \dots +d_{k+\frac{(i+1)(i+2)}{2}}.
   \end{align*}
   It then follows that
   \begin{align*}
(e_{i+1} - e_i ) - (e_i -e_{i-1}) &=  d_{k-i}+ d_{k+\frac{i(i+1)}{2}+1}+ d_{k+\frac{i(i+1)}{2}+2} + \dots +d_{k+\frac{(i+1)(i+2)}{2}}
\\ & - \left ( d_{k-i+1}+ d_{k+\frac{i(i-1)}{2}+1}+ d_{k+\frac{i(i+1)}{2}+2} + \dots +d_{k+\frac{i(i+1)}{2}} \right )
   \end{align*}
   There are $i+2$ terms with a positive sign and $i+1$ with a negative sign. We can pair off the $i+1$ largest positive terms with smaller (in absolute value) negative terms to conclude the proof, as follows:
    \begin{align*}
(e_{i+1} - e_i ) - (e_i -e_{i-1}) &=  d_{k-i}+\left (d_{k+\frac{i(i+1)}{2}+1} -d_{k+\frac{i(i+1)}{2}} \right ) + \left (d_{k+\frac{i(i+1)}{2}+2} -d_{k+\frac{i(i+1)}{2}-1} \right ) 
\\& + \dots + \left (d_{k+\frac{(i+1)(i+2)}{2}-1} -d_{k+\frac{i(i-1)}{2}+1} \right ) + \left (d_{k+\frac{(i+1)(i+2)}{2}} -d_{k-i+1}\right )
\\& >0.
   \end{align*}
\end{proof}

%\begin{Theorem}
%    For all sufficiently large $n \in \mathbb N$, there exists a convex set $A \subset \mathbb R$ with cardinality $n$ such that $C_M(A) \ll \sqrt n$.
%\end{Theorem}

\begin{proof}[Proof of Theorem \ref{thm:main3}]

For each $j =1, \dots, n$, define
\[
a_j:=j(2n)^n+(j-1)(2n)^{n-1} + \dots + 2(2n)^{n-(j-2)}+ (2n)^{n-(j-1)}.
\]
   The set $A= \{ a_j : 1 \leq j \leq n \}$ is a convex set. Indeed, the consecutive differences of $A$ are given by
   \[
   d_j= a_{j+1} - a_j = (2n)^n+(2n)^{n-1} +  \dots  + (2n)^{n-j},
   \]
   a sequence which is strictly increasing.

   Let $M \subset A \times A$ be a matching such that $A-_M A$ is convex. Our goal is to prove that $|M| \ll \sqrt n$.

   Let $k \leq n-1$ be an integer. Repeating notation used earlier in the paper, set $d_j^{(k)}:= a_{j+k}-a_j$ and
   \[ 
   D_k= \{ d_j^{(k)} : 1 \leq j \leq n-k \}.
 \]
 We calculate that
 \begin{align*} \label{dcalc}
d_j^{(k)}&=k[(2n)^n+ (2n)^{n-1} + \dots + (2n)^{n-j}] 
\\ &+ (k-1) (2n)^{n-(j+1)} +  (k-2)(2n)^{n-(j+2)} + \dots +(2n)^{n-(j+k-1)}.
 \end{align*}
 An important feature of this construction is that the diameter of the components $D_k$, which is approximately $(2n)^{n-2}$, is significantly smaller than the gaps between consecutive components, which is approximately $(2n)^{n}$. 
 
 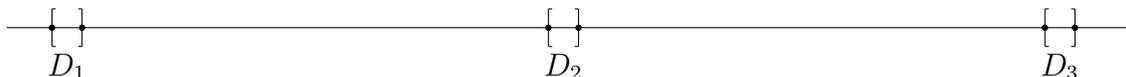
\begin{figure}[ht]

\centering

\begin{tikzpicture} %[scale=2]
  % Draw the number line
  \draw (6,0) -- (21,0);
  % Add labels for the values
  \node[below] at (6.8,-0.2) {$D_1$};
   \node[below] at (13.4,-0.2) {$D_2$};
    \node[below] at (20,-0.2) {$D_3$};
%  \node[below] at (1.1,0) {11000};
 % \node[below] at (1.11,0) {11100};
 % \node[below] at (1.111,0) {11110};
 % \node[below] at (1.1111,0) {11111};
 % \node[below] at (12,0) {100000}; 
  % Mark the values on the number line
  \draw[fill] (6.6,0) circle (1pt);
  \draw[fill] (7,0) circle (1pt);
 \draw[fill] (13.2,0) circle (1pt);
 \draw[fill] (13.6,0) circle (1pt);
  \draw[fill] (19.8,0) circle (1pt);
 \draw[fill] (20.2,0) circle (1pt);
\draw (6.6,0.25) -- (6.6,-0.25);
\draw (7,0.25) -- (7,-0.25);
\draw (6.6,0.25) -- (6.65,0.25);
\draw (6.65,-0.25) -- (6.6,-0.25);
\draw (7,0.25) -- (6.95,0.25);
\draw (6.95,-0.25) -- (7,-0.25);

\draw (13.2,0.25) -- (13.2,-0.25);
\draw (13.6,0.25) -- (13.6,-0.25);
\draw (13.2,0.25) -- (13.25,0.25);
\draw (13.25,-0.25) -- (13.2,-0.25);
\draw (13.6,0.25) -- (13.55,0.25);
\draw (13.6,-0.25) -- (13.55,-0.25);

\draw (19.8,0.25) -- (19.8,-0.25);
\draw (20.2,0.25) -- (20.2,-0.25);
\draw (19.8,0.25) -- (19.85,0.25);
\draw (19.85,-0.25) -- (19.8,-0.25);
\draw (20.2,0.25) -- (20.15,0.25);
\draw (20.2,-0.25) -- (20.15,-0.25);
\end{tikzpicture}

\caption{This diagram illustrates how the gaps between the consecutive $D_i$ are significantly larger than the diameters of the individual $D_i$. This is the heuristic reason why Claims \ref{claim:most21} and \ref{claim:weak} are valid.}

\end{figure}

 This allows us to conclude that, with at most one exception, a convex set can have at most one representative from each $D_k$. This is formalised in the following claim.

 \begin{Claim} \label{claim:most21}
     Suppose that $S \subset A-A$ is convex. Then there exists at most one $k \in \mathbb N$ such that $|S \cap D_k| \geq 2$. Moreover, if $|S \cap D_{k_2}| \geq 2$ then $S \cap D_{k_1}= \emptyset$ for all $ k_1 < k_2$.
 \end{Claim}

 \begin{proof}
    The first sentence of the claim follows from the second, and so it is sufficient to prove only the second sentence. Suppose for a contradiction that $k_1<k_2$, $|S \cap D_{k_1}| \geq 1$ and $|S \cap D_{k_2}| \geq 2$. Let $d_j^{k_2}$ be the smallest element of $S \cap D_{k_2} $. Since $S \cap D_{k_1}$ is non-empty, it follows that $d_j^{(k_2)}$ is not the first element of $S$, and since $S$ also contains a larger element of $D_{k_2}$, we also know that $d_j^{(k_2)}$ is not the last element of $S$. Let $x$ be the element of $S$ preceding $d_j^{(k_2)}$ and let $y$ be the element of $S$ following $d_j^{(k_2)}$. By the convexity of $S$,
     \begin{equation} \label{xandy}
     d_j^{(k_2)}-x < y- d_j^{(k_2)}.
     \end{equation}
     On the other hand
     \begin{align} \label{xbound}
     d_j^{(k_2)}-x \geq d_1^{(k_2)} - d_{n-(k_2 -1)}^{(k_2-1)} &> k_2 [(2n)^n+(2n)^{n-1}] - (k_2-1) [ (2n)^n + \dots + (2n) ] \nonumber
     \\& = (2n)^n + (2n)^{n-1} - (k_2-1)[ (2n)^{n-2} + \dots + (2n) ] \nonumber
     \\& >  (2n)^n  +  (2n)^{n-1} - n[ (2n)^{n-2} + \dots + (2n) ] \nonumber
     \\& \geq (2n)^n  +  (2n)^{n-1}  - n \cdot 2 \cdot  (2n)^{n-2} = (2n)^n  .
     \end{align}
     The second inequality uses the fact that $k_2 \leq n$, while the third inequality is an application of the inequality
     \begin{equation} \label{basic}
         (2n)^{j}+(2n)^{j-1}+ \dots + (2n) \leq 2 \cdot (2n)^j,
     \end{equation}
     which is valid for all $j, n \in \mathbb N$.
     
     Meanwhile, since $y \in D_{k_2}$, a similar calculation yields
     \begin{equation} \label{ybound}
     y -  d_j^{(k_2)} \leq d_{n-k_2}^{(k_2)}- d_1^{(k_2)} \leq k_2[ (2n)^{n-2} +  (2n)^{n-3} + \dots + (2n) ] \leq (2n)^{n-1}.
     \end{equation}
      Comparing \eqref{xbound}, \eqref{ybound} and \eqref{xandy}, we observe a contradiction.
 \end{proof}

 Using the same basic fact about the blocks $D_k$ again, namely that the gap between consecutive blocks is significantly larger than their individual diameters, we now show that the blocks which contain elements of a convex set must occur in a weakly convex form. A set $S= \{s_1 < s_2 < \dots < s_n \} \subset \mathbb R$ is said to be \textit{weakly convex} if the inequality
 \[
s_{i+1} - s_i \geq s_i - s_{i-1}
\]
holds for all $i=2,\dots, n-1$. For a given set $S \subset A-A$, we define 
 \[
 K(S)= \{ k \in \mathbb N : |S \cap D_k| \geq 1 \}.
 \]

 \begin{Claim} \label{claim:weak}
Suppose that $S \subset A-A$ is convex. Then the set $K(S)$ is weakly convex.
 \end{Claim}

 \begin{proof} 
Suppose for a contradiction that $K(S)$ is not weakly convex. Then there exists three consecutive elements 
\[
d_{j_1}^{(k_1)}, d_{j_2}^{(k_2)}, d_{j_3}^{(k_3)} \in S
\]
such that $k_2-k_1 > k_3-k_2$. Since the $k_i$ are integers, it follows that
\begin{equation} \label{kstuff}
    2k_2-k_1-k_3 \geq 1.
\end{equation}
The difference between $d_{j_2}^{(k_2)}$ and $d_{j_1}^{(k_1)}$ is
\begin{align*}
d_{j_2}^{(k_2)} - d_{j_1}^{(k_1)}&\geq d_{1}^{(k_2)} - d_{n-k_1}^{(k_1)} 
\\& > k_2[ (2n)^n + (2n)^{n-1}] -k_1[(2n)^n+ (2n)^{n-1} + \dots + (2n)] 
\\& =(k_2-k_1)[ (2n)^n +(2n)^{n-1}]  -k_1[(2n)^{n-2}+ (2n)^{n-3} + \dots + (2n)] 
\end{align*}
Meanwhile, the next difference can be bounded by
\begin{align*}
d_{j_3}^{(k_3)} - d_{j_2}^{(k_2)}&\leq d_{n-k_3}^{(k_3)} - d_{1}^{(k_2)}
\\ &< k_3[(2n)^n+ (2n)^{n-1} + \dots + (2n)] - k_2[(2n)^n+ (2n)^{n-1}]
\\& =(k_3 -k_2)[(2n)^n+ (2n)^{n-1}] + k_3[(2n)^{n-2}+ (2n)^{n-3} + \dots + (2n)].
\end{align*}
However, by the convexity of $S$, we also have $d_{j_2}^{(k_2)} - d_{j_1}^{(k_1)} < d_{j_3}^{(k_3)}-d_{j_2}^{(k_2)}$. Combining this with the previous two inequalities and applying \eqref{kstuff} yields
\[
(k_3+k_1) [ (2n)^{n-2}+ (2n)^{n-3} + \dots + (2n)] > (2k_2 -k_3-k_1) [(2n)^n+ (2n)^{n-1}] \geq [(2n)^n+ (2n)^{n-1}].
\]
We then once again use inequality \eqref{basic} to obtain
\begin{equation} \label{cont}
(k_3+k_1) \cdot 2 \cdot (2n)^{n-2} > (2n)^n.
\end{equation}
Finally, note that $k_3+k_1 < 2n$. This holds because $k_3,k_1 \leq n-1$, as the sets $D_{k_i}$ are only defined within this range. Plugging this into \eqref{cont}, we obtain the contradiction
\[
2 \cdot (2n)^{n-1} > (2n)^n.
\]

\end{proof}

 Another useful feature of this construction is that the consecutive differences within the components $D_k$ shrink rapidly, which makes it difficult to find a large convex sets in $A-A$.
 
 \begin{figure}[ht]

\centering

\begin{tikzpicture} [scale=1.3]
  % Draw the number line
  \draw (0,0) -- (11.2,0);
  % Add labels for the values
  \node[below] at (0,0) {$d_1^{(1)}$};
  \node[below] at (10,0) {$d_2^{(1)}$};
  \node[below] at (11,0) {$d_3^{(1)}$};
  \node[above] at (11.1,0) {$d_4^{(1)}$};
 % \node[below] at (1.1111,0) {11111};
 % \node[below] at (12,0) {100000}; 
  % Mark the values on the number line
  \draw[fill] (0,0) circle (1pt);
  \draw[fill] (10,0) circle (1pt);
 \draw[fill] (11,0) circle (1pt);
 \draw[fill] (11.1,0) circle (1pt);
\end{tikzpicture}

\caption{In this picture, we zoom in to take a closer look at the way the elements of $D_k$ are distributed (here we consider the set $D_1$ with $n=5$). Crucially, the gaps between consecutive elements of $D_k$ shrink rapidly, with the conseutive differences resembling a geometric progression with a small common ratio. This picture can be used for a sketchy justification of Claims \ref{claim:most22} and \ref{claim:ap}}

\end{figure}
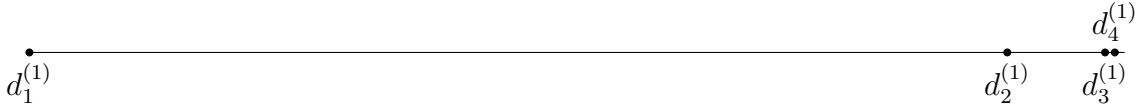

 We use this fact in the following claim to establish that a convex set cannot contain more than two elements from any $D_k$.

 \begin{Claim} \label{claim:most22}
     Suppose that $S \subset A-A$ is convex. Then
     \[
     |S \cap D_k| \leq 2
     \]
     holds for any $k=1,\dots,n-1$.
 \end{Claim}

\begin{proof}
  Suppose for a contradiction that there exist three consecutive elements of $S$ belonging to the same block $D_k$. In particular, we have $d_{i_1}^{(k)} < d_{i_2}^{(k)}< d_{i_3}^{(k)} $ satisfying
  \begin{equation} \label{convexbasic}
  d_{i_3}^{(k)} - d_{i_2}^{(k)} > d_{i_2}^{(k)} - d_{i_1}^{(k)}. 
  \end{equation}
  Then, since $i_2 \geq i_1+1$, we have
  \[
  d_{i_2}^{(k)} - d_{i_1}^{(k)} \geq d_{i_2}^{(k)} - d_{i_2-1}^{(k)} = (2n)^{n-i_2} + (2n)^{n-(i_2+1)} + \dots + (2n)^{n-(i_2+k-1)}> (2n)^{n-i_2}.
  \]
  On the other hand,
  \begin{align*}
  d_{i_3}^{(k)} - d_{i_2}^{(k)}  \leq d_{n-k}^{(k)} - d_{i_2}^{(k)} 
 &< k[ (2n)^{n-(i_2+1)} + \dots + (2n) ] 
  \\ & \leq k \cdot 2 \cdot  (2n)^{n-(i_2+1)} 
  \\& < n \cdot 2 \cdot  (2n)^{n-(i_2+1)} = (2n)^{n-i_2}.
  \end{align*}
  Combining the previous two inequalities with \eqref{convexbasic}, we obtain the intended contradiction

\end{proof}

%Let $M \subset A \times A$ be a matching. Suppose that $S \subset A-_M A$ is a convex set. Note that this is the first time in the proof where we use the fact that the convex set $S$ is derived from a matching.

By proving Claims \ref{claim:most21} and \ref{claim:most22}, we have essentially proved that each block of $D_k$ in $A-A$ can contain at most one element of a convex set $S \subset A-A$. We can be a little more precise; taking the potential exceptional block into account, we have the bound
\begin{equation} \label{Sbound}
|S| \leq |K(S)|+1.
\end{equation}
It remains to upper bound the size of the indexing set $K(S)$. We need one more claim to allow us to achieve this goal. Note that the following claim represents the first time in the proof where we use the fact that the convex set $S$ is derived from a matching.

\begin{Claim} \label{claim:ap}
    Suppose that $M \subset A \times A$ is a matching and that $S = A-_M A$ is a convex set. Then the indexing set $K(S)$ does not contain four consecutive elements which form an arithmetic progression.
\end{Claim}

\begin{proof}
    Suppose for a contradiction that four consecutive elements of $K(S)$ form an arithmetic progression. It then follows from Claim \ref{claim:most21} that there exist four consecutive elements $d_{j_1}^{(k)}, d_{j_2}^{(k+t)}, d_{j_3}^{(k+2t)}$ and $ d_{j_4}^{(k+3t)}$ in $S$, for some positive integers $k,t$ such that $k+3t \leq n - 1$. Since $S$ is derived from a matching, it must be the case that the $j_i$ are pairwise distinct. Write
    \begin{align*}
        e_1 & = d_{j_2}^{(k+t)} -  d_{j_1}^{(k)}
        \\  e_2 & = d_{j_3}^{(k+2t)} -  d_{j_2}^{(k+t)}
        \\  e_3 & = d_{j_4}^{(k+3t)} -  d_{j_2}^{(k+2t)}.
    \end{align*}
    Since $S$ is convex, we have $e_1 < e_2 < e_3$.

    We will now show that it must be the case that $j_2 < j_1$. Suppose for a contradiction that this is not true, and so $j_2>j_1$. Then
    \[
    e_1=d_{j_2}^{(k+t)} -  d_{j_1}^{(k)}  > t[(2n)^n+\dots + (2n)^{n-j_1}] + (t+1)(2n)^{n-(j_1+1)}.
    \]
    On the other hand
    \begin{align*}
    e_2 = d_{j_3}^{(k+2t)} -  d_{j_2}^{(k+t)} &\leq d_{n-(k+2t)}^{(k+2t)} -  d_{j_2}^{(k+t)}
    \\ & < t[ (2n)^{n} + \dots + (2n)^{n-j_2}] + (k+2t)[(2n)^{n-(j_2+1)}+ \dots + (2n)].
    \end{align*}
    It follows from the previous two bounds that
    \begin{align*}
    e_1 -e_2 &> (2n)^{n-(j_1+1)} - (k+2t)[(2n)^{n-(j_1+2)}+ \dots + (2n)] 
    \\& \geq (2n)^{n-(j_1+1)} - (k+2t) \cdot 2 \cdot  (2n)^{n-(j_1+2)} 
    \\& > (2n)^{n-(j_1+1)} - n \cdot 2 \cdot (2n)^{n-(j_1+2)}  =0
    \end{align*}
    This is a contradiction, and we have thus established that $j_2<j_1$. The exact same argument implies that $j_3<j_2$.

   Now, since $j_2< j_1$, we have
    \[\
    e_1=d_{j_2}^{(k+t)}- d_{j_1}^{(k)} \geq t[(2n)^n+ \dots + (2n)^{n-j_2}] - k[ (2n)^{n-(j_2+1)} + \dots + (2n)],
    \]
 Similarly, since $j_3<j_2$, it follows that
    \[
    e_2 = d_{j_3}^{(k+2t)} -  d_{j_2}^{(k+t)}  \leq t[ (2n)^n + \dots + (2n)^{n-j_3}] + (t-1) [ (2n)^{n-(j_3+1)}+ \dots + (2n)].
    \]
    Combining the previous two inequalities, and again making use of the fact that $j_3<j_2$, we have
    \begin{align*}
    e_1-e_2 &\geq (2n)^{n-(j_3+1)} - (k+t-1)[ (2n)^{n-(j_3+2)}+ \dots + (2n)]
    \\ &\geq  (2n)^{n-(j_3+1)} - (k+t-1) \cdot 2 \cdot  (2n)^{n-(j_3+2)}
    \\& > (2n)^{n-(j_3+1)} - n \cdot 2 \cdot  (2n)^{n-(j_3+2)} =0.
    \end{align*}
    This contradicts the fact that $e_1 <e_2$ and completes the proof of the claim.
\end{proof}

When we combine Claim \ref{claim:weak} and Claim \ref{claim:ap}, we see that the set $K$ is a weakly convex subset of $ \{ 1, \dots , n \}$ which does not contain four consecutive terms in an arithmetic progression. It follows that
\[
|K| \ll \sqrt n.
\]
Combining this with \eqref{Sbound}, the proof is complete.
   
\end{proof}

\section{Concluding remarks; sums instead of differences}

The problems considered in this paper were partly motivated by a potential application to a problem in discrete geometry concerning the minimum number of angles determined by a set of points in the plane in general position. This problem was considered recently in \cite{FKMPPW}, and similar problems can be traced back to the work of Pach and Sharir \cite{PS}. We found that progress on this question could be given by a solution to the following problem: given a convex set $A \subset \mathbb R$ estimate the size of the largest matching on $A$ which gives rise to a convex set in the image set $f(A,A)$, where $f: \mathbb R \times \mathbb R \rightarrow \mathbb R$ is a specific bivariate function whose rather complicated formula is omitted here. Theorems \ref{thm:main2} and \ref{thm:main3} of this paper solve this problem for this simplified case when $f(x,y)=x-y$.

With the potential application to the distinct angles problem in mind, an interesting future research direction could be to generalise the problems considered in this paper by considering an arbitrary $f: \mathbb R \times \mathbb R \rightarrow \mathbb R$ in place of the function $f(x,y)=x-y$. We conclude this paper with some remarks about the most natural case, whereby $f(x,y)=x+y$.

It is interesting to see that we can quite easily obtain an optimal result, giving a significant quantitative improvement to Theorem \ref{thm:main2}, if we consider sums instead of differences, as follows.

\begin{Theorem} \label{thm:main4}
Let $n \in \mathbb N$ and suppose that $A \subset \mathbb R$ is a convex set with cardinality $n$. Then there exists a matching $M \subset A \times A$ such that $|M| \geq \lfloor \frac{n}{2} \rfloor $ and $A+_M A$ is convex.
\end{Theorem}

\begin{proof}
 Suppose that $n$ is even and consider the matching
 \[
 M=\{(a_1, a_{n/2+1}), (a_2, a_{n/2+2}),..., (a_{n/2},a_n) \}.
 \]
 Then the set
 \[
 A+_M A= \{ a_{n/2+k}+a_k : k \in \{1, \dots , n/2 \} \}
 \]
 is convex. If $n$ is odd then we omit $a_n$ and use the same argument as above.
\end{proof}

In particular, it follows from Theorem \ref{thm:main4} that an analogue of the construction in the proof of Theorem \ref{thm:main3} is not possible if we take sums instead of differences. There are other cases in which problems concerning additive properties of convex sets are sensitive to  sums and differences. For instance, a construction in \cite{RNW} (see also \cite{Sch}) shows that there exists a convex set $A \subset \mathbb R$ and $x \in A-A$ with
\[
| \{ (a,b) \in A : a-b=x \}| \gg |A|.
\]
On the other hand, incidence geometry can be used to give the upper bound
\[
| \{ (a,b) \in A : a+b=x \}| \ll |A|^{2/3}
\]
for any convex $A \subset \mathbb R$ and $x \in A+A$.

We would also be interested to know whether Theorem \ref{thm:main} is still valid when $A-A$ is replaced by $A+A$. We were unable to prove anything non-trivial for this question.

\section*{Acknowledgements}

Krishnendu Bhowmick and Oliver Roche-Newton were supported by the Austrian Science Fund FWF Project P 34180. 
Ben Lund was supported by the Institute for Basic Science (IBS-R029-C1).
Part of this work was carried out at Vietnam Institute for Advanced Study in Mathematics (VIASM) in Hanoi, and we thank VIASM for their hospitality and the excellent working conditions.
We also thank Eyvindur Ari Palsson, Steven Senger and Audie Warren for helpful discussions.

\end{document}